\documentclass{article}       
\usepackage{amsmath, amsfonts, amssymb}
\usepackage{amsthm}

\newtheorem{defn}{Definition}[section]
\newtheorem{theo}[defn]{Theorem}
\newtheorem{lemma}[defn]{Lemma}
\newtheorem{prop}[defn]{Proposition}
\newtheorem{coro}[defn]{Corollary }
\theoremstyle{remark}
\newtheorem{remark}{Remark}[section]

\def\a{\alpha}
\def\b{\beta}
\def\c{\gamma}

\def\w{\omega}

\def\s{\sigma}
\def\t{\tau}
\def\I{I(P,u,v)}
\def\g{g_k(x,y,z)}
\def\f{f_k(x,y,z)}
 
\def\conmod#1#2{\equiv #1\ \hbox{$($mod\ }#2\hbox{$)$}}
\title{Proof of a Conjecture of Segre and Bartocci on Monomial Hyperovals in Projective Planes}

\author{Fernando Hernando\footnote{Inspire postdoctoral researcher at
the Claude Shannon Institute, funded by  Irish Research Council for Science, Engineering and Technology
adn Science Foundation Ireland Grant 06/MI/006, also partially supported by  MEC MTM2007-64704 (Spain).}
\\ Department of Mathematics\\ San Diego State University
 \\ United States\\
\\ Gary McGuire\footnote{Research supported by the Claude
Shannon Institute, Science Foundation Ireland Grant 06/MI/006.}
\\ School of Mathematical Sciences\\ University College Dublin \\Ireland\\
}

\begin{document}

\maketitle

\centerline{To Richard M.~Wilson on the occasion of his 65th birthday}

\bigskip\bigskip

\begin{abstract}
The existence of certain monomial hyperovals $D(x^k)$ in
the finite Desarguesian projective plane $PG(2,q)$, $q$ even,
is related to the existence of points on certain
projective plane curves $g_k(x,y,z)$.
Segre showed that some values of $k$  ($k=6$ and $2^i$) give rise to
hyperovals in $PG(2,q)$ for infinitely many $q$.  
Segre and Bartocci conjectured that these are the only values of $k$
with this property.
We prove this conjecture through the absolute irreducibility of the curves $g_k$.
\end{abstract}

\section{Introduction}

An {\it oval} in the finite Desarguesian
projective plane $PG(2,q)$ is a set of $q+1$ points with
the property that
no three points are collinear.  If $q$ is odd then such a set is maximal
with that property, and
a celebrated theorem of Segre (1955) states that all such
ovals are given algebraically by irreducible conics.
If $q$ is even however, the situation is more interesting.
Here, a set of points in $PG(2,q)$ of largest possible size
such that no three are collinear  has
cardinality $q+2$, and is called a {\it hyperoval}.

From now on in this paper we assume that $q$ is even.
A hyperoval can be constructed from a (nonsingular) conic by adjoining
the point at which all the tangents of the conic meet, the nucleus.
Such hyperovals are generally called regular hyperovals.
For $q>8$ there also exist irregular hyperovals which are not of the
form conic plus nucleus, see \cite{G1},\cite{G2},\cite{OKP} for example.

We represent the points of $PG(2,q)$ as homogeneous
triples with coordinates from $GF(q)$.
It is well known that all hyperovals can be written in the form
$$
\bigl\{ (1,x,f(x)) : x \in GF(q) \bigr\} \cup
\bigl\{(0,0,1),(0,1,0)\bigr\}
$$
where $f(x)$ is a polynomial with certain properties, see \cite{H},\cite{G2}.
Denote the above set by $D(f(x))$.
In this paper we shall examine the case where $f(x)$
is a monomial, say $f(x)=x^k$.
If $q=2^e$, Segre showed that the set $D(x^k)$ is a hyperoval
for the following
values of $k$ and the values of $e$ indicated:
\begin{eqnarray}
k=2^i,& \hbox{  when  $(i,e)=1$\ \ (\cite{S1}, 1957)} ,   \cr
k=6, &  \hbox{  when  $(2,e)=1$  (\cite{S2}, 1962)} .   
\end{eqnarray}

We wish to consider other values of $k$.  In particular,
we wish to consider the question of whether there are other
such infinite sequences, i.e.\ other fixed values
of $k$ for which $D(x^k)$ is a hyperoval for infinitely many  $q$.
This question was previously studied by Segre and Bartocci \cite{S3}.
Our main result is the following theorem, which was conjectured in \cite{S3}.
\bigskip
\begin{theo}\label{mainth}
For any fixed even positive integer $k$, if $k\not=6$ and $k\not=2^i$ then the set
$D(x^k)$ is  a hyperoval
in $PG(2,q)$ for at most a finite number of values of $q$.
\end{theo}

\bigskip
 
In \cite{M} the permutation properties of $1+x+\cdots +x^{k-1}$
on $GF(q)$ are studied.
It follows from \cite{LN} (p.505) that
this polynomial is a permutation polynomial
if and only if $D(x^k)$ is a hyperoval.
Hence our result sheds some light on this problem.
It is now trivial to see that $k$ must be even in order for
$D(x^k)$ to be a hyperoval, since
$1+x+\cdots +x^{k-1}$ maps both 0 and 1 to 1 if $k$ is odd.
So we assume $k$ is even from now on.
Values of $k$ which are functions of $e$ have been studied, see \cite{G1},
but we do not consider this here.

In case $D(x^k)$ is a hyperoval, we call it a monomial hyperoval
because $f(x)$ is a monomial.
Following \cite{H}, we will write $D(k)$ instead of $D(x^k)$.

We note some projective equivalences among these hyperovals.
If $D(k)$ is a hyperoval, then so is $D(m)$ where
$m=1/k,\ 1-k,\ 1/(1-k),\ (k-1)/k,\ k/(k-1)$, and everything is modulo $q-1$.
(If $(k,q-1)\neq 1$ or $(k-1,q-1)\neq 1$ then $D(k)$ is not a hyperoval.)
These hyperovals are all projectively equivalent, see \cite{H}.

In Section \ref{back} we shall prove the connection between the
polynomials $g_k(x,y,z)$ and the hyperovals $D(k)$.
Section \ref{singpts}  contains background information about algebraic curves,
and classifies the singular points of $g_k(x,y,z)$.
Section \ref{homcomint} computes the possible intersection multiplicities
of putative factors of $g_k$.
We shall completely factorize $\g$ for $k=2^i$ and $k=6$
in section \ref{segrerev}, and there
we will reprove Segre's theorems on these values.
In sections  \ref{part1proof}, \ref{part2proof}, \ref{part3proof}, we will use Bezout's theorem to prove
the main theorem in parts.
The result of section  \ref{part1proof}, the case $k\conmod24$,
was already proved by Segre and Bartocci \cite{S3}.
  
\section{Background}\label{back}

The set $D(k)$ being a hyperoval in $PG(2,q)$ is
equivalent to the  determinant
$$det \begin{pmatrix} 1&1&1            \cr
              x&y&z            \cr
             x^k&y^k&z^k        
\end{pmatrix}$$
being nonzero for all distinct $x,y,z \in GF(q)$.
Divide the determinant by $(x+y)(x+z)(y+z)$ and call
the resulting polynomial $g_k(x,y,z)$.
In other words, we define a binary polynomial $\g$ by
$$
\g := \frac{xy^k+yx^k+xz^k+zx^k+yz^k+zy^k }{(x+y)(x+z)(y+z)}.
$$
Our main theorem rests on the following, which is
also used by Segre and Bartocci.

\bigskip
\begin{theo}\label{implication}
If the polynomial $g_k(x,y,z)$ is absolutely
irreducible over $GF(2)$,
or has an absolutely irreducible factor defined over $GF(2)$,
then  $D(k)$ is
a hyperoval in $PG(2,q)$ for only a finite number of values of $q$.
\end{theo}

\proof
A form of the Weil bound due to W.\ Schmidt \cite{SC}
on the number of rational points on curves over finite
fields shows  that one can weaken the
hypothesis of nonsingularity of
a curve $f(x,y)$ over $GF(q)$
to absolute irreducibility,
and  still obtain a bound essentially the same as
that of Weil.

If the polynomial $g_k(x,y,z)$ of degree $k-2$ is absolutely
irreducible over $GF(2)$, then applying this form of Weil's
theorem  shows that the number $N_e$ of
(projective) rational points $(x,y,z)$
on $g_k(x,y,z)$ where $x,y,z\in GF(2^e)$ satisfies
$$
|N_e-2^e| < (k-3)(k-4)2^{e/2} + (k-2)^2   \eqno(1)
$$
for every $e$.  Once we show that the number of such rational points
where some of the coordinates are equal is at most $3k-2$, it will follow
that there are rational points over $GF(2^e)$ with $x,y,z$ distinct
for all $e$ sufficiently large.
Should it happen that 
 $g_k(x,y,z)$ is not absolutely
irreducible over $GF(2)$
but has an absolutely irreducible factor defined over $GF(2)$,
apply the same argument to this factor.

To this end we let $p(x,y,V)=g_k(x,y,x+V)$, and
note that projective points $(x,y,z)$ on $\g$ with $x=z$ are in $1-1$
correspondence with projective points $(x,y,0)$ on $p(x,y,V)$.
A simple computation (using the fact that $k$ is even) shows that
$$
p(x,y,0) = \frac{x^k + y^k }{(x+y)^2}.
$$
Again we let $q(x,W)=p(x,x+W,0)$, and note that
projective points $(x,y,z)$ on $\g$ with $x=z\neq y$ are in $1-1$
correspondence with affine points $(x,1)$ on $q(x,W)$.
Since $q(x,1)=x^k+(x+1)^k$, there are at most $k-1$
projective points $(x,y,z)$ on $\g$ with $x=z\neq y$.
A similar argument holds for points $(x,y,z)$ with
$x=y\neq z$ and $y=z\neq x$.  Counting the projective point
$(1,1,1)$ we get that there are at most $3k-2$ rational
points $(x,y,z)$ on $\g$ with $x,y,z$ not all distinct.
\qed
\bigskip

Remark: From the form of the Weil bound in $(1)$, we can actually
compute the value of $e$, say $e_0$, for which $N_e>3k-2$
for all $e\geq e_0$.  

Armed with this theorem, our task now is to
demonstrate the absolute irreducibility
of the polynomials $\g$ over $GF(2)$.
This is how we shall prove the results of this paper.

Bearing in mind Segre's results, $\g$ cannot be
absolutely irreducible when $k=2^i$ or $k=6$.

 Segre and Bartocci made the following conjecture (we are paraphrasing here).
 
 \bigskip
 
\textbf{Conjecture:} $g_k(x,y,z)$ has an absolutely irreducible factor over $\mathbb{F}_2$ for every even $k$ except $k=2^i$ and $k=6$.

\bigskip

We shall prove this conjecture in this paper. 
By Theorem \ref{implication} this is enough to prove our main theorem,
Theorem \ref{mainth}.

\section{Singular Points}\label{singpts}

It will be shown shortly that we are
allowed to work with the affine parts
of the homogeneous polynomials $\f$ and $\g$.
There will be no confusion if we use the same names, and so
\begin{eqnarray}
f_k(x,y) &:=& xy^k+yx^k+x^k+y^k+x+y        \cr
g_k(x,y) &:=& \frac{f_k(x,y)}{(x+y)(x+1)(y+1)},
\end{eqnarray}
and we consider the algebraic curves defined by these
polynomials over the algebraic closure of $GF(2)$.
Of course, $\g$ is absolutely irreducible if and only if $g_k(x,y)$
is absolutely irreducible.

For a polynomial $h$ and a point $P=(\a,\b)$, write
         $$h(x+\a,y+\b)=H_0(x,y)+H_1(x,y)+H_2(x,y)+\cdots$$
where each $H_i(x,y)$ is 0 or homogeneous of degree $i$.  If $m$ is the
smallest integer such that $H_m\neq0$ but $H_i=0$ for $i<m$, then $m$ is
called the {\it multiplicity of $h$ at $P$}, and is denoted by $m_P(h)$.
In particular, $P$ is on the curve associated to $h$ if and only if $m_P(h)\ge1$.
Also, by definition, $P$ is a singular point of $h$ if and only if $m_P(h)\ge2$.
The $m$ linear factors of $H_m$ are the tangent lines 
to $h(x,y)$ at $P$.
The collection of tangent lines is called the tangent cone.

The singular points can be found by equating the first
partial derivatives to zero.  We easily calculate ($k$ is even)
$$
\frac{\partial f_k}{ \partial x}(x,y) = y^k+1,  \qquad
\frac{\partial f_k }{ \partial y}(x,y) = x^k+1.
$$
Hence if $P=(\a,\b)$ is a singular point of $f_k(x,y)$,
then $\a$ and $\b$ are $k$-th roots of unity.
Write $k=2^i\ell$ where $\ell$ is odd and $i\ge1$.
Then $\a$ and $\b$ are $\ell$-th roots of unity, 
This proves

\begin{lemma}
$P=(\alpha,\beta)$ is a singular point of $f_k$ if and only if $\alpha^{\ell}=\beta^{\ell}=1$.
\end{lemma}

It follows that
$f_k(x,y)$ has $\ell^2$ singular points.  It is easy to check that
there are no singular points at infinity --- the three partial
derivatives of $\f$ are $x^k+y^k$, $x^k+z^k$, $y^k+z^k$,
and if these all vanish and $z=0$ then $x=y=0$ which is impossible.
This proves
\begin{lemma}
$f_k(x,y,z)$ has no singular points at  infinity.
\end{lemma}

Next we pin down the multiplicities of these singular points
$P=(\a,\b)$ on $f_k(x,y)$, and how things change for $g_k(x,y)$.
We compute that
\begin{eqnarray}
f_k(x+\a,y+\b)&=& \sum_{j=1}^k {k \choose j} \big( \a^{-j}x^jy+\b^{-j}y^jx
                 +(\b+1)\a^{-j}x^j +(\a+1)\b^{-j}y^j \big)  \cr
              &=&  F_0+F_1(x,y) + F_2(x,y) +\cdots    
\end{eqnarray}
using $\a^k=1=\b^k$.
Since $k\choose j$ is even for $1\leq j<2^i$ and odd for $j=2^i$,
and even for $j=2^i+1$,
we see that all singular points of $f_k(x,y)$ have multiplicity
$2^i$, except $(1,1)$ which has multiplicity $2^i+1$.  This
claim follows from
\begin{eqnarray}
F_0&=&\alpha^{2^{i}\ell}(\beta+1)+\beta^{2^i\ell}(\alpha+1)+\alpha+\beta,\cr
F_1(x,y)&=&\alpha^{2^{i}\ell}y+\beta^{2^i\ell}x+x+y, \cr
F_{2^i}(x,y) &=& (\b+1)\a^{-2^i}x^{2^i} + (\a+1)\b^{-2^i}y^{2^i},  \cr
F_{2^i+1}(x,y) &=& \a^{-j}x^{2^i}y + \b^{-j}y^{2^i}x.
\end{eqnarray}
 
 We classify the points into three types:
\begin{itemize}
\item[(I)] $P=(1,1)$.
\item[(II)] Either $\alpha=1$ or $\beta=1$ or $\alpha=\beta$.
\item[(III)] $\alpha\neq \beta$ and $\alpha\neq1\neq \beta$.
\end{itemize}

\begin{remark}
If $\ell=1$ and $i>1$ the only singular point is $(1,1)$.
\end{remark}

Defining $w(x,y):=(x+y)(x+1)(y+1)$ we note the following
multiplicities on $w$:  $m_P(w)=3$ if $P=(1,1)$ (Type I), 
$m_P(w)=1$ if $P$ has Type II, and
$m_P(w)=0$ for all other singular points $P=(\a,\b)$ of Type III.
At long last we arrive at the multiplicities for $g_k(x,y)$.

\begin{center}
\begin{tabular}{|l | c| c | c|}
\hline
Type & Number of Points & $m_P(f_k)$ & $m_P(g_k)$ \\
\hline
I & $1$ & $2^{i}+1$ & $2^{i}-2$\\
\hline
II & $3({\ell}-1)$ & $2^i$ & $2^i-1$\\
\hline
III & $ ({\ell}-1)({\ell}-2)$ & $2^i$ & $2^i$\\
\hline
\end{tabular}
\end{center}

There are $3(\ell-1)$  points of multiplicity $2^i-1$, and so
there are $(\ell-1)(\ell-2)$ singular points of multiplicity $2^i$
on $g_k(x,y)$.

\bigskip

Let $u$ and $v$ be projective plane curves over  $\overline{GF(2)}$;
we assume that $u$ and $v$ have no common component.
The {\it intersection multiplicity} $I(P,u,v)$ of $u$ and $v$
at $P$  is
the unique nonnegative
integer satisfying and determined by the seven properties
listed on pages 74--75 of  \cite{F}.
For our purposes there are two important properties.
One  is that
$I(P,u,v)\ne 0$ if and only if both $m_P(u)$ and $m_P(v)$ are $\ge 1$.
Another important property is that
$I(P,u,v)\geq m_P(u) m_P(v)$, with equality occurring
if and only if $u$ and $v$ do not have a common tangent at $P$
(their tangent cones are disjoint).

We will use the following theorem
from classical algebraic geometry, whose proof can be found in \cite{F}.

\begin{theo}[Bezout's Theorem]
Let  $u$ and $v$ be projective plane curves with
no common component. Then
$$\sum_{P} I(P,u,v)=(\hbox{\rm deg~} u)(\hbox{\rm deg~} v).$$
\end{theo}

Our method of proving absolute irreducibility will be to assume
that $g_k (x,y,z)$ is reducible,
say $g_k (x,y,z) = u(x,y,z)v(x,y,z)$, and obtain a contradiction by
applying Bezout's theorem to the curves $u$ and $v$.
If a point $P$ has $\I \ne 0$, then
$m_P(g_k) = m_P(u) + m_P(v) \ge 2$, and so $P$
is a singular point of $g_k (x,y,z)$.
We have seen that the projective curves $g_k (x,y,z)$ have
no singular points at infinity.
Therefore, since the only points  $P$ that
give a nonzero contribution to the sum in Bezout's theorem are
singular points of $g_k (x,y,z)$,
we may just work with the affine part of $g_k (x,y,z)$.

\section{Homogeneous Components and Intersection Multiplicity}\label{homcomint}

The following result is clear, because we are in characteristic 2.

\begin{lemma}\label{repeatedLine}
$F_{2^i}=(\s x+\t y)^{2^i}$ where $\s=(\alpha^{2^{i}(\ell-1)}(\beta+1))^{1/2^i}$
and $\t=(\beta^{2^i(\ell-1)}(\alpha+1))^{1/2^i}$
\end{lemma}

\begin{lemma}\label{differentfactors}
$F_{2^i+1}$ consists of $2^i+1$ different linear factors.
\end{lemma}
\begin{proof}
Consider $h(x)=F_{2^i+1}(x,1)=\alpha^{2^{i}(\ell-1)}x^{2^i}+\beta^{2^i(\ell-1)}x$. If $h(x)$ has a repeated root at $a$ then $h'(a)=0$. Consider the derivative $h'(x)=\beta^{2^i(\ell-1)}$ which is never zero and therefore there are no
repeated factors in $F_{2^i+1}$.
\end{proof}

Next we make a crucial observation for our proofs.
For the rest of this paper, we let $L=\s x+\t y$, so that $F_{2^i}=L^{2^i}$.
Suppose $g_k(x,y)=u(x,y)v(x,y)$, and suppose that the Taylor expansion at
a singular point $P=(\a,\b)$ is
\[
u(x+\a, y+\b)=L^{r_1}+u_1, \  v(x+\a,y+\b)=L^{r_2}+v_2
\]
where wlog $r_1\leq r_2$.
Then $F_{2^{i}+1}=L^{r_1}(v_1+L^{r_2-r_1}u_1)$. 
From Lemma \ref{differentfactors} we deduce that:
\begin{lemma}\label{multp0o1}
With the notation of the previous paragraph,
\begin{itemize}
\item[(i)] Either $r_1=1$ or $r_1=0$.
\item[(ii)]  If $r_1=1$ then $gcd(L,v_1+L^{r_2-r_1}u_1)=1$.
\end{itemize}
\end{lemma}

We next make two quick remarks to aid us in moving between
$f_k(x,y)$ and $g_k(x,y)$.
Recall the notation of section 3, and
suppose that $P=(\a,\b)\neq(1,1)$ is a singular point of $g_k(x,y)$
such that $F_{2^i}(x,y)\neq0$ at $P$.
To apply Proposition 2 to $g_k$ we need to know the greatest
common divisor $(G_{m}(x,y),G_{m+1}(x,y))$ where $m=m_P(g_t)$.
This can be found from $(F_{2^i}(x,y),F_{2^i+1}(x,y))$ as follows.

Again letting $w(x,y)=(x+y)(x+1)(y+1)$,
we have
            $$f_k(x+\a,y+\b)=w(x+\a,y+\b)g_k(x+\a,y+\b),$$
and so
$$         F_{2^i}(x,y)+F_{2^i+1}(x,y)+\cdots
          = (W_0+W_1(x,y)+\cdots)(G_m(x,y)+G_{m+1}(x,y)+\cdots).
$$
where polynomials with subscript $i$ are 0 or homogeneous of degree $i$.

\begin{remark}\label{remw0}
Here we assume $W_0\neq0$ which is equivalent to
saying that $P$ is a Type III point, and 
$m=2^i$.  Multiplying out and using (2) gives
\begin{eqnarray}
F_{2^i}&=&W_0 G_{2^i} = (\sigma x+\tau y)^{2^i} \cr
F_{2^i+1}&=&W_0 G_{2^i+1} + W_1 G_{2^i} ,
\end{eqnarray}
where $\sigma^{2^i}=(\b+1)\a^{-2^i}$ and
$\tau^{2^i}=(\a+1)\b^{-2^i}$.
It follows from these equations that
$(F_{2^i},F_{2^i+1})=(G_{2^i},G_{2^i+1})$.
\end{remark}

\begin{remark}\label{remw1}
Here we assume $W_0=0$ which is
equivalent to saying $P$ is a Type II point, and $m=2^i-1$.  As in Remark \ref{remw0} we get
\begin{eqnarray}
F_{2^i}&=&W_1 G_{2^i-1} = (\sigma x+\tau y)^{2^i} \cr
F_{2^i+1}&=&W_1 G_{2^i} + W_2 G_{2^i-1} .
\end{eqnarray}
It is clear that (up to scalars) $W_1 =\sigma x+\tau y$,
and so $(F_{2^i},F_{2^i+1})= \sigma x+\tau y$ because
$F_{2^i+1}(x,y)$ has distinct linear factors (Lemma \ref{differentfactors}).
Hence $(G_{2^i-1},G_{2^i})=1$.
\end{remark}

\bigskip

The next result will help us to compute intersection multiplicities.

\begin{prop}\label{GCD=1}
 Let $h(x,y)$ be an affine curve.
Write $h(x+\alpha,y+\beta)=H_m+H_{m+1}+\cdots$ where
$P=(\alpha,\beta)$ is a point on  $h(x,y)$ of multiplicity $m$.
Suppose that  $H_m$ and $H_{m+1}$  are relatively prime,
and that there is only one tangent direction at $P$.
If $h=uv$ is reducible, then $I(P,u,v)=0$.
 \end{prop}
\begin{proof}
See \cite{Janwa-McGuire-Wilson}
\end{proof}

\subsection{Type I}

We upper bound the intersection multiplicity at the Type I point.

\begin{lemma}
If $g_k(x,y)=u(x,y)v(x,y)$ and $P=(1,1)$ then
$I(P,u,v)\leq  (2^{i-1}-1)^2$.
\end{lemma}
\begin{proof}
Let $P$ be of Type I.
We know that $m_P(g_k)=2^i-2=m_P(u)+m_P(v)$.
From Lemma \ref{differentfactors} we know that
$F_{2^i+1}$ has $2^i+1$ different linear factors. 
Thus, $I(P,u,v)=m_p(u)m_p(v)$. 
This quantity is maximized when  $m_P(u)=m_P(v)$ and in this case
 $m_p(u)m_p(v)=(2^{i-1}-1)^2$.
\end{proof}

\subsection{Type II}

We show that intersection multiplicities at Type II points are 0, so these
points may be disregarded.

\begin{lemma}\label{IntersectionTypeII}
If $g_k(x,y)=u(x,y)v(x,y)$ and $P=(\alpha,\beta)$ is a point of type (II) then
$I(P,u,v)=0$.
\end{lemma}
\begin{proof}
There are three kinds of Type (II) point.
\begin{itemize}
\item If $P=(\alpha,1)$ then $F_{2^i}=(\alpha+1)y^{2^i}$. Hence $gcd(F_{2^i},F_{2^i+1})=y$ and therefore
$gcd(G_{2^i-1},G_{2^i})=1$ by Remark \ref{remw1}. 
The proof concludes using Proposition \ref{GCD=1}.
\item If $P=(1,\beta)$ use the same argument with $gcd(F_{2^i},F_{2^i+1})=x$.

\item If $P=(\alpha,\alpha)$ use the same argument with  $gcd(F_{2^i},F_{2^i+1})=x+y$.
\end{itemize}
\end{proof}

\subsection{Type III}

We show that there are two possibilites for the intersection multiplicity at a Type III point.

\begin{lemma}\label{IntersectionTypeIII}
Let $k=2^i \ell$ where $\ell$ is odd.
If $g_k(x,y)=u(x,y)v(x,y)$ and $P=(\alpha,\beta)$ is a point of type (III) then
either $I(P,u,v)= 2^i$ or $I(P,u,v)= 0$.
\end{lemma}
\begin{proof}
Assume $g_k(x,y)=u(x,y)v(x,y)$. 
Since $P$ is not in $w(x,y)=(x+1)(y+1)(x+y)$ by Lemma \ref{multp0o1} we know that $m_P(u)$ is either $1$ or $0$. If $m_P(u)=0$ then $I(P,u,v)= 0$.
If $m_P(u)=1$ we proceed as follows.

Let $L(x,y)=\s x + \t y$ and suppose 
we have the following Taylor expansions at $P$:
$$
u(x+\a, y+\b)=L(x,y)+U_2(x,y)+\cdots
$$
$$
v(x+\a, y+\b)=L(x,y)^{2^i-1}+V_{2^{i}}(x,y)+\cdots
$$
It follows that
$$
u(x+\a,y+\b)L(x,y)^{2^i-2}+v(x+\a,y+\b)=L(x,y)^{2^i-2}U_2(x,y)+V_{2^{i}}(x,y)+\cdots.
$$
By definition of intersection multiplicity we have 
\[
I(P,u,v)=I(0,u(x+\a,y+\b),u(x+\a,y+\b)L(x,y)^{2^i-2}+v(x+\a,y+\b))
\] 
so we compute the righthand side.   
Notice that $L(x,y)\nmid L(x,y)^{2^i-2}U_2(x,y)+V_{2^{i}}(x,y)$ because $L(x,y)(L(x,y)^{2^i-2}U_2(x,y)+V_{2^{i}(x,y)})=G_{2^{i}+1}(x,y)$ and $G_{2^{i}+1}(x,y)$ may contain
$L(x,y)$ at most one time. Therefore,  $u(x+\a,y+\b)$ and $u(x+\a,y+\b)L^{2^i-2}+v(x+\a,y+\b)$ have different tangent cones.
It follows from a property of $I(P,u(x,y),v(x,y))$ that 
\[
I(0,u(x+\a,y+\b),u(x+\a,y+\b)L^{2^i-2}+v(x+\a,y+\b))=
\]
\[
 m_0(u(x+\a,y+\b))m_0(u(x+\a,y+\b)L^{2^i-2}+v(x+\a,y+\b))=2^i.
\]
\end{proof}

\section{Segre Revisited.}\label{segrerev}

In this section we study the polynomials $g_k(x,y)$ when
$k=2^i$ and $k=6$.  First let us examine $k=2^i$.
\begin{eqnarray}
f_k(x+1,y+1) &= (x+1)^{2^i}y + (y+1)^{2^i}x + x+y  \cr
             &= x^{2^i}y+y^{2^i}x  \cr
             &= xy \prod_{\c\in GF(2^i)^*} (x+\c y)  .
             \end{eqnarray}
Replace $x$ by $x+1$, $y$ by $y+1$, and divide by $(x+y)(x+1)(y+1)$
to get

\begin{theo}
When $k=2^i$ we have the following factorization,
$$
g_k(x,y) = \prod_{\c \in GF(2^i)\backslash \{0,1\} }  (x+\c y+\c +1).
$$
\end{theo}

\begin{coro}[Segre]
When $k=2^i$ the set $D(k)$ is a hyperoval in $PG(2,2^e)$
if and only if $(i,e)=1$.
\end{coro}

\proof
We have to show that $g_k(x,y)$ has
the necessary rational points over $GF(2^e)$
if and only if $(i,e)>1$.

Suppose that $i$ and $e$ are relatively prime and that
there exists $a,b\in GF(2^e)$
with $a\neq b, a\neq 1,b\neq1$ such that $g_k(a,b)=0$.
By the factorization above, there exists $\c\in GF(2^i)\backslash \{0,1\} $
such that $a+\c b+\c+1=0$.  But this implies $\c=(a+1)/(b+1)
\in GF(2^i) \cap GF(2^e) = GF(2)$, a contradiction.

Conversely suppose $(i,e)>1$, and choose $a,b$ distinct
in  $GF(2^i) \cap GF(2^e)$ but not in $GF(2)$.
Letting $\c=(a+1)/(b+1)$ shows that $g_k(a,b)=0$, again using the factorization.
\qed

We remark that $P=(1,1)$ is the only singular point in this
case, and it has multiplicity $2^i-2$.

\bigskip
Next we consider $k=6$.
 Here  
$$
g_6(x,y)=y^4+y^3(1+x)+y^2(1+x+x^2)+y(1+x+x^2+x^3)+1+x+x^2+x^3+x^4.
$$
It is easy to show that $g_6$ must be absolutely irreducible, or must factor
over $GF(4)$ into absolutely irreducible factors.
If $GF(4)=\{ 0,1,\omega ,\omega^2 \} $,
then in fact $g_6=AB$ where
$$
A(x,y)=1+\w x+x^2+(\w+\w x)y+y^2
$$
and its conjugate
$$
B(x,y)=1+\w^2 x+x^2+(\w^2+\w^2 x)y+y^2.
$$
 
We have proved:

\begin{theo}
When $k=6$ we have the factorization $g_k(x,y)=A(x,y)B(x,y)$,
where $A(x,y)$ and $B(x,y)$ are absolutely irreducible
and are given above.
\end{theo}

\begin{coro}[Segre]
When $k=6$ the set $D(k)$ is a hyperoval in $PG(2,2^e)$
if and only if $(2,e)=1$.
\end{coro}

\proof
We have to show that $g_6(x,y)$ has
the necessary rational points over $GF(2^e)$
if and only if $e$ is even.

If $e$ is even, then $g_6(\w^2,\w)=0$.  Done.

Suppose now that $e>1$ is any odd  integer.
We claim that $A(x,y)$ and $B(x,y)$
have no rational points over $GF(2^e)$.
For suppose that $A(a,b)=0$ where $a,b\in GF(2^e)$.
Visibly we can assume $(a,b)\neq(0,0)$.  Then
$$
b^2+b\w+ab\w+a^2+a\w+1=0,
$$
and provided $a+b+ab\neq0$ this implies
$\w=(a+b+1)^2/(a+b+ab)\in GF(4)\cap GF(2^e)$,
which is a contradiction.  But if $a+b+ab=0$
then $a+b+1=0 \Rightarrow ab=1 \Rightarrow 1+b^{-1}+b=0
\Rightarrow 1+b+b^2=0$ which is impossible.
Similarly for $B(x,y)$.
\qed

Of course, these results can be proved in other ways.

\section{The case $k\conmod24$ }\label{part1proof}

This case occurs when $k=2\ell$.  

\begin{theo}\label{KnownCase}
If  $k\conmod24$ and $k>6$ then $g_k(x,y)$ is absolutely irreducible.
\end{theo}

\begin{proof}
 Assume that $g_k(x,y)=u(x,y)v(x,y)$.
Let $P=(\alpha,\beta)$ be a singular point, depending of its type we have that:
\begin{itemize}
\item If $P=(1,1)$ then $m_p(g_k)=0$, so $I(P,u,v)=0.$
\item If $P$ is a point of type II from Lemma \ref{IntersectionTypeII} we know that 
$I(P,u,v)=0.$
\item Suppose $P=(\alpha, \beta)$ has Type III.
We want to apply Proposition \ref{GCD=1} to $g_k(x+\alpha,y+\beta)=F_{2}+F_3+\cdots$
where 
$$
F_{2}=\alpha^{2(\ell-1)}(\beta+1)x^{2}+\beta^{2(\ell-1)}(\alpha+1)y^{2}.
$$
$$
F_{3}=\alpha^{2(\ell-1)}x^{2}y+\beta^{2(\ell-1)}y^{2}x=xy(\alpha^{2(\ell-1)}x+\beta^{2(\ell-1)}y).
$$
We clearly have that $gcd(xy,F_2)=1$. If  $gcd(\alpha^{2(\ell-1)}x+\beta^{2(\ell-1)}y,F_2)=1$ 
then by Proposition \ref{GCD=1} we have $I(P,u,v)=0$.
The only way that this $gcd\neq 1$ is that $\alpha^{2(\ell-1)}x+\beta^{2(\ell-1)}y\mid F_2$, and this occurs if
$\alpha^{4(\ell-1)}=\alpha^{2(\ell-1)}(\beta+1)$ and $\b^{4(\ell-1)}=\b^{2(\ell-1)}(\a+1)$. Equivalently,
$1=\alpha^{2}(\beta+1)$ and $1=\b^{2}(\a+1)$. Adding both equations we get $\a+\b=\a\b$, multiplying then we 
get $1=\a^2\b^2(\a\b+\a+\b+1)$. Substituting $\a+\b=\a\b$ in $1=\a^2\b^2(\a\b+\a+\b+1)$ we get $\a^2\b^2=1$
which implies $\a\b=1$.
So, $gcd= 1$ unless $\a\b=1$ and $\a+\b=1$. 

Suppose we have a point $P$ with  $\a\b=1$ and $\a+\b=1$.
 Then $\a$ and $\b$ are roots of $x^2+x+1$,
and so they lie in $GF(4)$.  This means there can be at most two such points.
By Lemma  \ref{IntersectionTypeIII}, at those points $P$ we have $I(P,u,v)=0$ or 2, and therefore
  $\sum_P I(P,u,v)=0$ or $2$ or $4$.
  But if $k\conmod24$  and $k>6$ then $g_k$ has degree at least 8, so it is impossible that
  $(\deg u)(\deg v)\leq 4$.  Thus we get a  contradiction to  Bezout's theorem.  
\end{itemize}
\end{proof}
\qed

Note that the proof fails when $k=6$, as it should, because 
 $\sum_P I(P,u,v)=4$ and the two factors of $g_6$ have degree 2.

\section{The case $g_k(x,y)$ irreducible over $\mathbb{F}_2$}\label{part2proof}

Here is a well known result.

\begin{lemma}\label{IfIrreducibleEqualDegreeFactors}
Suppose that $p(\underline{x})\in \mathbb{F}_q[x_1,\ldots,x_n]$ is of degree $t$ and is irreducible in $\mathbb{F}_q[x_1,\ldots,x_n]$. There there exists $r\mid t$ and an absolutely irreducible polynomial
$h(\underline{x})\in\mathbb{F}_{q^r}[x_1,\ldots,x_n]$ of degree $\frac{t}{r}$ such that
$$
p(\underline{x})=c\prod_{\sigma\in G}\sigma(h(\underline{x})),
$$
where $G=Gal(\mathbb{F}_{q^r}/\mathbb{F}_{q})$ and $c\in \mathbb{F}_{q}$. Furthermore if $p(\underline{x})$
is homogeneous, then so is $h(\underline{x})$.
\end{lemma}

\begin{remark}\label{samedegree}
Notice that if $u(x,y)=\sum a_{i,j}x^i y^j$ then $\sigma(u(x,y))=\sum \sigma(a_{i,j})x^i y^j$ where $\sigma\in G$
is  the Frobenius map (or a power of it). Therefore, $u$ and $\sigma(u)$ have the same monomials and only differ in some coefficients. This means that both  $u$ and $\sigma(u)$ have the same degree.
\end{remark}

\begin{theo}
If $g_{k}(x,y)$ is irreducible over $\mathbb{F}_2$ then $g_{k}(x,y)$ is absolutely irreducible for every $k$ but $k=6$ and $k=2^i$.
\end{theo}
\begin{proof}
Suppose not, then $g_{k}(x,y)=u(x,y)v(x,y)$. Using Remark \ref{samedegree} we have that $deg(u)=deg(v)=2^{i-1}\ell-1$ .
We apply Bezout's Theorem to $u$ and $v$:
\begin{equation}\label{eq1}
\sum_{P\in Sin(g)} I(P,u,v)=deg(u) deg(v)=(2^{i-1}\ell-1)^2.
\end{equation}
We can bound the left hand side as follows,
\begin{equation}\label{eq2}
\sum_{P\in Sing(g_k)} I(P,u,v)=\sum_{P\in I} I(P,u,v)+\sum_{P\in II} I(P,u,v)+\sum_{P\in III} I(P,u,v)
\end{equation}
\begin{equation}\label{eq3}
\leq (2^{i-1}-1)^2+2^i(l-1)(l-2).
\end{equation}
We have that (\ref{eq1})$\leq$(\ref{eq3}) by Bezout's Theorem, 
so if we prove that (\ref{eq1})$>$(\ref{eq3}) we get a contradiction.
The inequality  (\ref{eq1})$>$(\ref{eq3})  is
$$
2^{2i-2}\ell^2-2^i\ell+1>2^i\ell^2-2^i(3\ell)+2^i+2^{2i-2}+1
$$
$$
(2^{2i-2}-2^i)\ell^2+2^{i+1}\ell-2^i-2^{2i-2}>0
$$
$$
(2^{2i-2}-2^i)(\ell^2-1)+2^{i+1}(\ell-1)>0.
$$

The question now is when the left-hand side is positive. 
Clearly when $\ell=1$ it is not positive. If $\ell>1$ and
$(2^{2i-2}-2^i)\geq 0$ then it is clearly positive. If $(2^{2i-2}-2^i)<0$ then it could be positive or negative.
Clearly $(2^{2i-2}-2^i)<0$ when $i=0$ or $i=1$.

If $i=0$, we have the inequality $\frac{3}{4}\ell^2-2\ell+\frac{5}{4}<0$. The solutions of the equation $\frac{3}{4}\ell^2-2\ell+\frac{5}{4}=0$ are $1$ and $5/3$. Thus for $i=0$ and $\ell>1$ the polynomial $g_k(x,y)$ is absolutely irreducible over $\mathbb{F}_2$.

If $i=1$ the possible $\ell$ for which the latter equation is not positive are those solutions of the equation:
$$
-\ell^2+1+4\ell-3=-\ell^2+4\ell-3=0.
$$
The possible solutions are $\ell=1$ or $\ell=3$. Hence, we conclude that except for $k=6$ ($i=1$ and $\ell=3$)
and $k=2^i$ ($\ell=1$) the polynomial $g_k(x,y)$ is absolutely irreducible over $\mathbb{F}_2$.
\end{proof}

\section{Case $g_k(x,y)$ not irreducible over $\mathbb{F}_2$}\label{part3proof}

Suppose $g_k=f_1\cdots f_r$ is the factorization into irreducible factors over $\mathbb{F}_2$.
Let $f_j=f_{j,1}\cdots f_{j,n_j}$ be the factorization of $f_j$ into $n_j$ absolutely irreducible factors.
Each $f_{j,s}$ has degree $\deg (f_j)/n_j$.

\begin{lemma}\label{OnlyTwohasMultiplicity}
If $P$ is a point of type $III$ then one of the following  holds:
\begin{enumerate}
\item  $m_P(f_{j,s})=0$ for all $j\in \{1,\ldots,r\}$ and $s\in\{1,\ldots,n_j\}$
except for a  pair $(j_1,s_1)$  with  $m_P(f_{j_1,s_1})=2^i$.
\item
$m_P(f_{j,s})=0$ for all $j\in \{1,\ldots,r\}$ and $s\in\{1,\ldots,n_j\}$
except for two  pair $(j_1,s_1)$ and $(j_2,s_2)$ with $m_P(f_{j_1,s_1})=1$ and $m_P(f_{j_2,s_2})=2^{i}-1$.
\end{enumerate}
\end{lemma}
\begin{proof}
This is a consequence of Lemma \ref{repeatedLine} and Lemma \ref{multp0o1}.
Consider $u=f_{a,b}$ and $v=\prod_{j\neq a, s\neq b}f_{j,s}$ from Lemma \ref{multp0o1}
we know that $m_P(f_{a,b})$ is either $0$ or $1$ or $2^{i}-1$ or $2^i$ (resp $m_p(v)$ is either $2^i$ or $2^i-1$ or $1$ or $0$). But this is true for any pair $(a,b)$.

Clearly no two components $f_{a,b}$ and $f_{a',b'}$ has multiplicity greater than or equal to $2^i-1$ because the total multiplicity $m_P(g_k)=2^i$. And there are no two components $f_{a,b}$ and $f_{a',b'}$ with multiplicity equal to $1$,
because then $u=f_{a,b}f_{a',b'}$ has $L$ two times in the tangent cone and $v=g/u$ has $L^{2^{i-2}}$ in the tangent cone which is impossible. Hence the only possibilities are:
\begin{itemize}
 \item[(i)] There exists $(a,b)$ with $m_P(f_{a,b})=2^i$, and  $m_P(f_{j,s})=0$ for
$(j,s)\neq(a,b) $.
\item[(ii)] There exist $(a,b)$ and $(a',b')$ with $m_P(f_{a,b})=1$ and $m_P(f_{a',b'})=2^i-1$, and $m_P(f_{j,s})=0$ for
$(j,s)\neq(a,b) $ , $(j,s)\neq(a',b') $.
\end{itemize}
\end{proof}

\begin{lemma}\label{IntersectionReducibleI}
If $P$ is a point of type $I$, then for any two components $f_{a,b}$ and $f_{a',b'}$ we have that $I(P,f_{a,b},f_{a',b'})=m_P(f_{a,b})m_P(f_{a',b'})$.
\end{lemma}
\begin{proof}
From Lemma \ref{differentfactors} the tangent cones of $f_{a,b}$ and $f_{a',b'}$ has no common factors.
\end{proof}

\begin{lemma}\label{IntersectionReducibleII}
If $P$ is a point of type $II$, then for any two components $f_{a,b}$ and $f_{a',b'}$ we have that $I(P,f_{a,b},f_{a',b'})=0$.
\end{lemma}
\begin{proof}
Consider $u=f_{a,b}$ and $v=g_m/u$. From Lemma \ref{IntersectionTypeII} we know that $I(P,u,v)=0=\sum_{(j,s)\neq (a,b)} I(P,u,f_{j,s})$, then $I(P,f_{a,b},f_{a',b'})=0$.
\end{proof}

\begin{lemma}\label{IntersectionReducibleIII}
Let $P$ is a point of type $III$ and $g_k(x,y)=\prod_{j=1}^r\prod_{s=1}^{n_j} f_{j,s}$. The intersection multiplicity  $I(P,f_{a,b},f_{a',b'})$ of any two components $f_{a,b}$ and $f_{a',b'}$  is either $0$ or $2^{i}$.
\end{lemma}
\begin{proof}
Consider $u=f_{a,b}$ and $v=g_m/u$. From Lemma \ref{IntersectionTypeII}  we know that either $I(P,u,v)=0=\sum_{(j,s)\neq (a,b)} I(P,u,f_{j,s})$, then $I(P,f_{a,b},f_{a',b'})=0$ or
$I(P,u,v)=2^i=\sum_{(j,s)\neq (a,b)} I(P,u,f_{j,s})$ using Lemma \ref{OnlyTwohasMultiplicity} we have that there exits $(a',b')$ with $I(P,f_{a,b},f_{a',b'})=2^i$.
\end{proof}

We need some more technical results  for the main theorem,
which give us some upper bounds.

\begin{lemma}\label{Cota} \ \ \ \ \   \ \ \ \ \ \ \  \ \ \ \ \ \
\begin{itemize}
\item[(i)] If $g_k(x,y)$ does not have an absolutely irreducible factor over $\mathbb{F}_2$, then,
\begin{equation}\label{eq10}
\sum_{j=1}^r \deg(f_j)^2/n_j < deg(g_k)^2/2.
\end{equation}
\item[(ii)]
\begin{eqnarray*}
\sum_{j=1}^r \sum_{1\leq i<s\leq n_j} \sum_{\substack{P \in Sing(g_k)\\ 
P\neq (1,1)}}  I(P,f_{j,i},f_{j,s})+
\sum_{1\leq j<l\leq r} \sum_{\substack{1\leq i\leq n_j\\1\leq s\leq n_l}}
\sum_{\substack{P \in Sing(g_k)\\ P\neq (1,1)}}  I(P,f_{j,i},f_{l,s})\\\leq 2^i (\ell-1)(\ell-2)
\end{eqnarray*}
\item[(iii)]
\begin{eqnarray*}
\sum_{j=1}^r \sum_{1\leq i<s\leq n_j}  \sum_{\substack{ P= (1,1)}}  I(P,f_{j,i},f_{j,s})+
\sum_{1\leq j<l\leq r} \sum_{\substack{1\leq i\leq n_j\\1\leq s\leq n_l}}
\sum_{\substack{ P= (1,1)}}   I(P,f_{j,i},f_{l,s}))\\
\leq (2^{i-1}-1)(2^i-3)
\end{eqnarray*}
\end{itemize}
\end{lemma}
\begin{proof}
\begin{itemize}
\item[(i)]
$$\sum_{j=1}^r \deg(f_j)^2/n_j\leq  \sum_{j=1}^r \deg(f_j)^2/2=1/2(deg(f_1)^2+\cdots+deg(f_r)^2)\leq
1/2 deg(g_k)^2$$
\item[(ii)] From Lemma \ref{IntersectionReducibleII} we know that if $P$ is a point of type II then $I(P,f_{j,i},f_{l,s})=0$ for every $j,l\in\{1,\ldots,r\}$ and $1\leq i\leq n_j$,$1\leq s\leq n_l$. From Lemma \ref{IntersectionReducibleIII} we now that for each point P of type III there is at most two components $f_{a,b}$ and $f_{a',b'}$ for which $I(P,f_{a,b},f_{a',b'})=2^i$ and zero otherwise. Taking into account that there are $(\ell-1)(\ell-2)$ points of type III we get the result.
\item[(iii)] From Lemma \ref{IntersectionReducibleI} we have that if $P$ is a point of type $I$, then for any two components $f_{a,b}$ and $f_{a',b'}$ we have $I(P,f_{a,b},f_{a',b'})=m_P(f_{a,b})m_P(f_{a',b'})$. Hence we have to prove the following,
  \begin{eqnarray*}
\sum_{j=1}^r \sum_{1\leq i<s\leq n_j}   m_P(f_{j,i})m_P(f_{j,s})+
\sum_{1\leq j<l\leq r} \sum_{\substack{1\leq i\leq n_j\\1\leq s\leq n_l}}
   m_P(f_{j,i})m_P(f_{j,s})\\\leq (2^{i-1}-1)(2^i-3).
\end{eqnarray*}
    Notice that the left hand side is a maximum when $m_P(f_{j,s})=1$ for every $j\in\{1,\ldots,r\}$ , $s\in \{1,\ldots,n_j\}$.  The latter equation is

\[
\sum_{j=1}^r \sum_{1\leq i<s\leq n_j}   m_P(f_{j,i})m_P(f_{j,s})+
\sum_{1\leq j<l\leq r} \sum_{\substack{1\leq i\leq n_j\\1\leq s\leq n_l}}m_P(f_{j,i})m_P(f_{j,s})
\]
\[
\leq \sum_{j=1}^r \sum_{1\leq i<s\leq n_j} 1+ \sum_{1\leq j<l\leq r} \sum_{\substack{1\leq i\leq n_j\\1\leq s\leq n_l}} 1 \]  
\[  =\binom{2^i-2}{2}=(2^i-2)(2^i-3)/2
=(2^{i-1}-1)(2^i-3).
\]
\end{itemize}
\end{proof}

Finally, here is our main result.

\begin{theo}\label{MainTheo}
$g_k(x,y)$ always has an absolutely irreducible factor over $\mathbb{F}_2$.
\end{theo}
\begin{proof}
We apply Bezout's Theorem one more time  to
the product
$$f_1f_2\ldots f_r=(f_{1,1}\ldots f_{1,n_1})(f_{2,1}\ldots f_{2,n_2})\ldots (f_{r,1}\ldots f_{r,n_r}).
$$
The sum of the intersection multiplicities can be written
\[
\sum_{j=1}^r \sum_{1\leq i<s\leq n_j} \sum_{P \in Sing(g_k)}  I(P,f_{j,i},f_{j,s})+
\sum_{1\leq j<l\leq r} \sum_{\substack{1\leq i\leq n_j\\1\leq s\leq n_l}}
\sum_{P \in Sing(g_k)}  I(P,f_{j,i},f_{l,s})
\]
where the first term is for factors within each $f_j$, and the second term
is for cross factors between $f_j$ and $f_l$.
Using Lemma \ref{Cota}, part (ii) and (iii), the previous sums can be bounded by
\begin{equation}\label{eq17}
\leq (2^{i-1}-1)(2^i-3)+2^i(l-1)(l-2).
\end{equation}

On the other hand,
the right-hand side of Bezout's Theorem is
\begin{equation}\label{bez6}
\sum_{j=1}^r \sum_{1\leq i<s\leq n_j} \deg(f_{j,i})\deg(f_{j,s})+
\sum_{1\leq j<l\leq r} \sum_{\substack{1\leq i\leq n_j\\1\leq s\leq n_l}}\deg(f_{j,i})\deg(f_{l,s}).
\end{equation}
Since each $f_{j,s}$ has the same degree for all $s$,
the first term is equal to
\[
\sum_{j=1}^r  \deg(f_j)^2\ \frac{n_j-1}{2n_j}=
\frac{1}{2} \sum_{j=1}^r  \deg(f_j)^2 - \frac{1}{2} \sum_{j=1}^r  \frac{\deg(f_j)^2}{n_j}.
 \]
 Note that
\begin{eqnarray*}
(\deg (g_k))^2&=&\biggl( \sum_{j=1}^r  \deg (f_j) \biggr)^2\\
&=&  \sum_{j=1}^r \deg (f_j)^2 +2 \biggl( \sum_{1\leq j<l\leq r} \deg (f_j) \deg (f_l) \biggr)\\
&=&   \sum_{j=1}^r \deg (f_j)^2 +2 \sum_{1\leq j<l\leq r} \biggl( \sum_{s=1}^{n_j} \deg (f_{j,s})\biggr)
  \biggl( \sum_{i=1}^{n_l} \deg (f_{l,i})\biggr)\\
  &=&  \sum_{j=1}^r \deg (f_j)^2 +2
  \sum_{1\leq j<l\leq r} \sum_{\substack{1\leq i\leq n_j\\1\leq s\leq n_l}}\deg(f_{j,i})\deg(f_{l,s}).
\end{eqnarray*}
Substituting both of these into (\ref{bez6}) shows that (\ref{bez6}) is equal to
\begin{equation}\label{otherside}
\frac{1}{2}\biggl(\deg(g_k)^2-\sum_{j=1}^r  \frac{\deg(f_j)^2}{n_j} \biggr).
\end{equation}
Using (\ref{eq10}) we get
\begin{equation}\label{othersidecota}
\frac{1}{2}\biggl(\deg(g_k)^2-\sum_{j=1}^r  \frac{\deg(f_j)^2}{n_j} \biggr)>
\frac{1}{2}\biggl(\deg(g_k)^2-\deg(g_k)^2/2 \biggr)=\deg(g_k)^2/4.
\end{equation}

Comparing (\ref{othersidecota}) and (\ref{eq17}),
so far we have shown that Bezout's Theorem implies the following inequality:
 \[
\deg(g_k)^2/4 \leq (2^{i-1}-1)(2^i-3)+2^i(\ell-1)(\ell-2).
 \]
 Let us now show that the opposite is true, to get a contradiction.  Suppose
  \[
\deg(g_k)^2/4 > (2^{i-1}-1)(2^i-3)+2^i(\ell-1)(\ell-2).
 \]
 Then
  \[
(2^{2i-2}\ell^2-2^i\ell+1) > (2^{i-1}-1)(2^i-3)+2^i(\ell^2-3\ell +2).
 \]

\[
2^{2i-2}(\ell^2-2) > 2^i\ell^2-2 2^i\ell-2^{i-1}.
\]

\[
2^{2i-2}(\ell^2-2) > 2^i(\ell^2-2 \ell+1)-32^{i-1}.
\]

 \[
2^{2i-2}(\ell^2-2) > 2^i(\ell-1)^2-32^{i-1}.
\]

If $\ell>1$: the latter equation is equivalent to $2^{i-2} > \frac{(\ell-1)^2}{(\ell^2-2)}-\frac{3}{2(\ell^2-2)}$. One can easily see that  for $\ell>1$ we have that $(\ell-1)^2<(\ell^2-2)$. Thus, $\frac{(\ell-1)^2}{(\ell^2-2)}-\frac{3}{2(\ell^2-2)}<1$ and $2^{i-2}\geq 1$ iff $i\geq 2$.
Only left to study cases $i=1$. If $i=1$ we apply Theorem \ref{KnownCase}.

If $\ell=1$:  $-2^{2i-2} > -3 2^{i-1} \Leftrightarrow 2^{i-2} < 3/2$. For $i\geq 3$ this is clearly not true. So the only doubts arise for $i=1,2$. If $i=1$, $1/2<3/2$, and if $i=2$ $1<3/2$, therefore when $\ell=1$ and $i\geq 1$ we do not have a contradiction.

 \end{proof}

\end{document}